\documentclass[11pt,reqno]{amsart}
\newtheorem{theorem}{Theorem}[section]
\newtheorem{corollary}[theorem]{Corollary}

\newtheorem{lemma}[theorem]{Lemma}

\theoremstyle{remark}
\newtheorem{remark}[theorem]{Remark}

\usepackage{bbm}
\usepackage{graphicx}
\usepackage{verbatim}
\usepackage{color}
\definecolor{deprecated}{rgb}{0.5,0.5,0.5}

\newcommand{\cus}{S_{c}}
\newcommand{\tcus}{T_{c}}
\newcommand{\tcusn}{T_{c}^{n}}
\newcommand{\cT}{T}

\newcommand{\cJ}{\mathcal{J}}
\newcommand{\cF}{\mathcal{F}}
\newcommand{\cFt}{\mathcal{F}_{t}}
\newcommand{\cS}{\mathcal{S}}

\newcommand{\Exp}{{\sf E}}

\newcommand{\Pro}{{\sf P}}

\newcommand{\ind}[1]{\mathbbm{1}_{\{#1\}}}   %indicator

\begin{document}
\title[Optimal sequential change-detection]{Optimal sequential change-detection for fractional diffusion-type processes}
\author{Alexandra Chronopoulou}
\address{Alexandra Chronopoulou \\  Department of Statistics and Applied Probability\\ University of California, Santa Barbara \\ CA
93106, USA}

\email{chronopoulou@pstat.ucsb.edu}
\author{Georgios Fellouris}
\address{Georgios Fellouris \\  Department of Mathematics\\ University of Southern California\\ 
%3620 South Vermont Ave., KAP 416, 
Los Angeles, CA 90089, USA}
\email{ fellouri@usc.edu}

\begin{abstract}
We consider the problem of  detecting an abrupt change in the distribution of a sequentially observed stochastic process. 
We establish the optimality of the CUSUM  test with respect to a modified version of Lorden's criterion for arbitrary processes with continuous paths
%, where detection delay is measured in terms of the quadratic variation of the log-likelihood ratio between the post and pre--change distribution that is accumulated between the time of the change and the time an alarm is raised. 
and apply this general result to the special case of fractional diffusion--type processes.
% and obtain an explicit sufficient condition for the optimality of CUSUM in this context. 
As a by-product, we show that the CUSUM test optimizes Lorden's \textit{original} criterion  when a fractional Brownian motion with Hurst index $H$ 
adopts a polynomial drift term with exponent  $H+1/2$ after the change. 
%which is known to hold when a standard Brownian motion adopts a linear drift after the change, to 
%In this way, we generalize the well-known Lorden-optimality of the CUSUM test in the case that a standard Brownian motion adopts a linear drift.  
% whose exponent depends on $H$. % a novel family of model-free, Lorden-like criteria is introduced and it is shown that these criteria are  optimized by the \textsc{cusum} test when a fractional Brownian motion adopts a polynomial drift. Finally, %Finally, a modification of the continuous-time \textsc{cusum} test is proposed for the case that only discrete-time observations  are available. 
\end{abstract}

 % insert keywords separated by a semicolon


\maketitle

\section{Introduction}
The quick  detection of an abrupt change in the behavior of a stochastic system is an important problem in many application areas, such as quality control,  target tracking, navigation,  seismology, bio--surveillance, computer security. %(see for example \cite{mac}, \cite{nvk}, \cite{tarkim}). 
More specifically, the problem is to find a detection rule that raises an alarm, as soon as possible after the change has occurred, based on sequential observations of the system. Thus, a good detection rule should have small detection delay, but also a low frequency of false alarms in its repeated applications. 

There are three main formulations for the change--detection problem that  balance the trade--off 
between these two antithetic goals. A Bayesian approach, developed by Shiryaev \cite{shiry}, where the change--point is modeled as a random variable, and two minimax approaches, due to Pollak \cite{pol} and Lorden \cite{lord}, where the change--point is considered to be an unknown, deterministic parameter. For a comparison of these formulations we refer to the paper of Moustakides \cite{moustann} and for exhaustive treatments of sequential change--detection 
to the books of Basseville and Nikiforov \cite{niki} and Hadjiliadis and Poor \cite{oly}.  

In the present work, we focus on  Lorden's approach, which has a deep connection with the so-called Cumulative Sums (CUSUM) test, a detection structure that was proposed by Page \cite{page} and has been very popular in applications since then (see for example \cite{olw}). Lorden \cite{lord} quantified the performance of a detection rule with its worst (with respect to the time of the change) conditional expected detection delay given the worst possible scenario until the time of the change and suggested the minimization of  this criterion subject to an upper bound on the rate of 
false alarms.  In the case of independent and identically distributed observations before and after the change, Lorden showed that this criterion is attained asymptotically by the CUSUM test. Moustakides \cite{moust} --and later Ritov \cite{rito}-- proved that the CUSUM test is an \textit{exact} solution to Lorden's optimization problem. This exact optimality property was extended in continuous--time in the case of a standard Brownian motion that adopts a linear drift
(Shiryaev \cite{shircus}, Beibel \cite{beib}), as well as in the case of a diffusion--type process  with constant ``signal to  noise ratio'' (Moustakides \cite{moustito}). The latter was a by--product of the main result obtained in the same paper, according to which  the CUSUM procedure optimizes a modified version of Lorden's criterion in the much more general framework of diffusion--type processes that satisfy a ``full--energy'' condition.

In the present work, we extend the optimality of CUSUM (with respect to the criterion introduced in \cite{moustito}) to arbitrary \textit{continuous--path} processes and apply this general result to the case of \textit{fractional} processes of diffusion--type.
This is a class of continuous--path processes that do not admit a semimartingale decomposition, thus they are clearly not in the scope of existing optimality results. We show in particular that the CUSUM test is optimal (in the above sense) for detecting a change from a fractional Brownian motion (fBm) to a fractional Ornstein--Uhlenbeck process (fOU), as well as for detecting the emergence of a  linear drift in a fBm for \textit{any} value of the Hurst index. 

Finally, we establish the  optimality of the CUSUM test with respect to Lorden's \textit{original} criterion when the observed process is a 
fBm with Hurst index $H$ that adopts a polynomial drift term with exponent $H+1/2$ after the change. This is one of the rare instances for which the exact solution to this problem is known. Of course, in the special case $H=1/2$, we recover the  optimality of CUSUM for the detection 
of a linear drift in a  standard Brownian motion.

Apart from a theoretical point of view, these extensions are also interesting due the increasing presence of fBm in a wide range of applications. 
This  can be explained by the fact that fBm is well--suited to model phenomena that are characterized by self--similarity and/or long--memory. 
As a result, it has been used  as the basic building block for models in a variety of fields, such as hydrology, traffic networks, finance  and economics (see for example \cite{beran}, \cite{CV},  \cite{CR}, \cite{Hurst}). 

These diverse applications have also triggered a great interest in the statistical inference for  processes related to fBm. 
Indeed, Kleptsyna and Le Breton  in  \cite{KLB1} studied the maximum likelihood estimator  (MLE) for the parameter 
in the drift of the fractional Ornstein--Uhlenbeck  process with Hurst index $H>1/2$. 
Tudor and Viens \cite{TV} used Malliavin calculus techniques to study the MLE for any $H \in (0,1)$ and a more general class of fractional diffusions, 
in which the drift coefficient is linear with respect to the unknown parameter. For the same class of processes,  Rao \cite{raoestim} studied a sequential version of the MLE. The problem of sequential testing for fractional diffusion--type processes was considered by Rao in \cite{rao}. 

%All these results were based on an integral transformation of the observed process, which was introduced by Norros et al. \cite{NVV} 
%in order to compute the MLE for the parameter in the linear drift of a fBm. This is the technique we also use in the present paper in order to establish the %optimality property of the CUSUM test in the context of fractional diffusion--type processes. 

In what follows, we establish the optimality of the CUSUM test for arbitrary  continuous--path stochastic processes in Section 2,
we focus on fractional diffusion--type processes in  Section 3 and we conclude in Section 4.

% and as a by-product we extend the optimality  of CUSUM with respect to Lorden's original criterion. 
%we consider the sequential change detection problem for fractional stochastic differential equations. In Section 5, we propose a novel class of Lorden-like, model-free criteria and  show that they are optimized by  the \textsc{cusum} test under certain dynamics. In Section 6, we discuss the implementation of the \textsc{cusum} test in the case of discrete-time observations and we conclude in Section 7. 

%-----------------------------------------------------------
%-----------------------------------------------------------
%-----------------------------------------------------------

\section{CUSUM optimality for continuous-path processes}
\subsection{Problem formulation}
Let $(\Omega, \cF)$ be the canonical space of continuous functions on $[0, \infty)$ that vanish at 0. %, with the associated Borel $\sigma$--algebra. 
We denote by $\{\xi_{t}\}$ the coordinate process on this space and  by $\{\cFt\}$ its natural filtration, thus 
$\xi_{t}(\omega):=\omega(t)$ for every $\omega \in \Omega$ and $\cFt:=\sigma(\xi_{s}: 0 \leq s \leq t)$ for every $t>0$, whereas 
$\xi_{0}:=0$ and $\cF_{0}:=\{\emptyset, \Omega\}$. Let $\Pro_{0}, \Pro_{\infty}$ be two completely specified probability measures on $(\Omega, \cF)$. We assume that $\Pro_{0}$ and $\Pro_{\infty}$ are locally equivalent, i.e. mutually absolutely continuous when they are restricted to the $\sigma$--algebra $\cFt$ for any  $0 < t< \infty$,  and we denote by $u_{t}$ the corresponding log--likelihood ratio 
%given the observed path up to time $t$, 
\begin{equation*} %\label{llr}
u_{t} := \log \frac{\text{d}\Pro_{0}}{\text{d}\Pro_{\infty}} \Big|_{\cFt}, \quad t>0 \, ; \quad u_{0}:=1.
\end{equation*}

We assume that the distribution of $\{\xi_{t}\}$, which we denote by $\Pro_{\tau}$, changes at some unknown, deterministic time $\tau \in [0,\infty]$ from $\Pro_{0}$ to $\Pro_{\infty}$. Thus, $\Pro_{\tau}$ coincides with $\Pro_{\infty}$ on $\cFt$ for any $t \in [0, \tau]$ and with $\Pro_{0}$ on $\cFt$ for any $t \in (\tau, \infty)$. Then, it is clear that  $\Pro_{\tau}$ is mutually absolutely continuous with $\Pro_{\infty}$ on  $\cFt$  for any  $t > \tau$ and that
\begin{equation*} 
u_{t}-u_{\tau}= \log \frac{\text{d}\Pro_{\tau}}{\text{d}\Pro_{\infty}} \Big|_{\cFt}, \quad t > \tau.
\end{equation*}
%Assuming that the distributions $\Pro_{0}, \Pro_{\infty}$ are known, but the time of the change $\tau$ completely unknown, 
The goal is to find a  sequential detection rule, that is an $\{\cFt\}$--stopping time $T$, that minimizes the detection delay for any change point $\tau  \in [0,\infty)$.
%so that $(T-\tau)^{+}$ takes small values under $\Pro_{\tau}$, while $T$ takes large values under $\Pro_{\infty}$. 
Since it is not possible to do so for every $\tau$, we follow a minimax approach and consider the following constrained optimization problem:
\begin{align} \label{moust_crit} 
\begin{split}
&\inf_{T}  \cJ_{M}[\cT] \quad \text{when} \quad  \frac{1}{2} \, \Exp_{\infty} [\langle u \rangle_{T}] \geq \gamma, \; \text{where} \\
\cJ_{M}[\cT] &:= \sup_{ \tau \geq 0} \; \text{esssup}  \; \frac{1}{2} \,  \Exp_{\tau} [\bigl(\langle u \rangle_{\cT}-\langle u \rangle_{\tau} \bigr)^{+} \, | \; {\mathcal{F}}_{\tau}],
\end{split}
\end{align}
where $\langle u \rangle_{t}$ is the quadratic variation of $u_{t}$. When the observed process is of diffusion-type before and after the change, this problem coincides with the criterion introduced by Moustakides in \cite{moustito} (and this explains the inclusion of the  redundant $1/2$ in its definition).  When $\langle u \rangle_{t}$ is proportional to $t$, it is equivalent to  Lorden's  criterion, 
\begin{align}   \label{lord_crit}
\begin{split}
&\inf_{\cT} \cJ_{L}[\cT] \;\text{when} \quad \Exp_{\infty} [\cT] \geq \gamma, \quad \text{where} \\  
\cJ_{L}[\cT] &:= \sup_{ \tau \geq 0} \; \text{esssup}  \; \Exp_{\tau} [(\cT-\tau)^{+} | \; {\mathcal{F}}_{\tau}],
\end{split}
\end{align}
In both formulations, the goal is to find a detection rule that minimizes the worst (with respect to $\tau$) conditional expected \textit{detection delay}  given the worst possible history of observations up to the time of the change, subject to a lower bound on the \textit{period of false alarms}. Their difference is in the way ``detection delay'' and ``period of false alarms'' are measured. This is done in terms of the actual time in (\ref{lord_crit}) and 
in  terms of the accumulated quadratic variation of the log-likelihood ratio process in (\ref{moust_crit}). 

The latter approach also has an appealing interpretation in terms of Kullback--Leibler information, which becomes clear with (\ref{klfap}) and (\ref{kldel}) and  justifies calling (as in \cite{moustito}) $\cJ_{M}[T]$ the \textit{K-L detection divergence} and $\frac{1}{2} \Exp_{\infty}[\langle u \rangle_{T}]$ the \textit{K-L false alarm divergence} of the detection rule $T$. However, the main advantage of this formulation is that its solution is known for a much larger class of dynamics, as we discuss below. For simplicity, in what follows we will say that a detection rule is $\cJ_{M}$\textit{--optimal} if it solves the problem defined in (\ref{moust_crit}) and $\cJ_{L}$\textit{--optimal} if it solves the problem defined in (\ref{lord_crit}).

\subsection{Main result}
Let us first define the CUSUM stopping time
%This is for example the case when $\xi_{t} =W_{t} + (t-\tau)^{+}$, that is when the observed process is a standard Brownian motion that adopts a linear drift, in which case $u_{t}= \xi_{t}- t/2$ and $\langle u \rangle_{t}=t/2$.  \subsection{The CUSUM test}
%, which was originally proposed by Page \cite{page} for the sequential detection of changes in discrete--time processes, can be defined in continuous--time in the following way:
\begin{equation*} %\label{cusum}
\cus := \inf \{t \geq 0: y_{t} \geq c\}, \quad y_{t}:= u_{t}- \inf_{0 \leq s \leq t} u_{s}, \quad t \geq 0,
\end{equation*}
where the threshold $c$ is assumed to satisfy the corresponding false alarm constraint with equality, that is 
$\Exp_{\infty}[\langle u \rangle_{\cus}]=\gamma$ for the problem defined in (\ref{moust_crit}) and $\Exp_{\infty}[\cus]=\gamma$ for the problem defined in (\ref{lord_crit}).

As we discussed in the Introduction, it is known that this detection rule is $\cJ_{M}$--optimal for diffusion-type processes that satisfy a  ``full--energy'' condition and additionally $\cJ_{L}$--optimal for diffusion-type processes that have a constant ``signal to noise ratio''.  With the following theorem we extend these optimality properties.

\begin{theorem}\label{MAIN} 
The CUSUM test is $\cJ_{M}$-optimal, if the following condition is satisfied
\begin{align}
& \lim_{t \rightarrow   \infty} \langle u \rangle_{t} =\infty  \quad \Pro_{0}, \Pro_{\infty}-\text{a.s.}  \label{full}
\end{align}
When in particular  $\langle u \rangle_{t}$ is proportional to $t$, the CUSUM test is also $\cJ_{L}$--optimal. 
%solves Problem B.
%is $\cJ_{M}$--optimal, i.e. for any $\gamma>0$ it is
%\begin{equation}   \label{lord_crit}
%\cJ_{M}[\cus] = \inf_{\cT \in \cC_{M}(\gamma)}  \cJ_{M}[\cT] =e^{c}-c-1,
%\end{equation} 
%where $c$ is the unique root of the equation $\gamma= e^{c}-c-1$.
%In particular, if $\gamma= e^{c}-c+1$, then $\cus$ solves  (\ref{moust_crit}). 
\end{theorem}

%\begin{align}
%& y_{t}:= u_{t}- \inf_{0 \leq s \leq t} u_{s}, \quad t \geq 0, \\
%& g(x):=e^{-x}+x-1, \quad x \geq 0 \\
%& \tcus:=\cT \wedge \cus = \min(\cT,\cus) ,\\
%& \tcCg=\{ T \in \cC: \Exp_{\infty}[ \langle u \rangle_{T}] = \gamma\} .
%& \cJ'_{M}[\cT] := \frac{\Exp_{\infty}[e^{y_{\tcus}} g(y_{\tcus})]}{\Exp_{\infty}[e^{y_{\tcus}}]} .
%\end{align}

%Suppose that the observed process $\xi$ has continuous paths and that condition (\ref{A2}) is satisfied,  %The \textsc{cusum} test $\cus$ solves problem (\ref{moust_crit}), among  stopping times $\cT$ such that  $\Exp_{0}[A_{\cT}] , \Exp_{\infty}[A_{\cT}] < \infty$, %as long as the threshold $c$ is chosen so that $\gamma= e^{c}-c-1$. The CUSUM test is \begin{enumerate} 
%for any $\tau \in [0,\infty]$ and $c>0$,  $\Pro_{\tau}(\cus<\infty)=1$, i.e. the CUSUM test terminates almost surely. 
%\subsection{Proof of Theorem \ref{MAIN}}

In order to prove this theorem, we start with the following lemma, which reveals the structure of the log-likelihood ratio process $\{u_{t}\}$.

\begin{lemma} \label{L0}
There exist continuous processes  $\{\tilde{X}_{t}\}$ and $\{X_{t}\}$, which are local martingales (vanishing at 0) with respect to $\Pro_{\infty}$ and $\Pro_{0}$ respectively and have the same quadratic variation, so that
\begin{align*} %\label{lr}
u_{t} &= \tilde{X}_{t} -\frac{1}{2} \,  \langle u \rangle_{t} = -X_{t}+\frac{1}{2} \,  \langle u \rangle_{t}, \quad t \geq 0.
\end{align*}
%where $\tilde{X}_{t} = -X_{t}+\langle X \rangle_{t}$ and $\langle u \rangle_{t}=\langle \tilde{X} \rangle_{t} =   \langle X \rangle_{t}$. 
% Moreover,  $\tilde{X}_{t}= \frac{1}{2} \,  \langle X \rangle_{t}
%If additionally $\Pro_{0}$ and $\Pro_{\infty}$ are singular measures, then 
%\begin{align} 
%& \Pro_{0} \Bigl( \lim_{t \rightarrow   \infty} \langle X \rangle_{t} =\infty \Bigr)=1,  \label{full1}\\
%& \Pro_{\infty} \Bigl( \lim_{t \rightarrow   \infty} \langle X \rangle_{t} =\infty \Bigr)=1.  \label{full2}
%\end{align}
\end{lemma}

\begin{proof}
Since the likelihood ratios $\{e^{u_{t}}\}$ and $\{e^{-u_{t}}\}$ are continuous martingales with respect to $\Pro_{\infty}$ and $\Pro_{0}$ respectively, 
it is well--known (see for example, Proposition (1.6), pg. 328 in \cite{yor}) that there exist two \textit{unique}, continuous, local martingales $\{\tilde{X}_{t}\}$ and $\{X_{t}\}$  with respect to $\Pro_{\infty}$ and $\Pro_{0}$ respectively,  so that 
\begin{equation*}
e^{u_{t}} = \exp \{ \tilde{X}_{t}- (1/2) \, \langle \tilde{X} \rangle_{t}\}, \quad  e^{-u_{t}} = \exp \{ X_{t} - (1/2) \, \langle X \rangle_{t}\},
\end{equation*}
where  $X_{0}=\tilde{X}_{0}=0$, since $u_{0}=1$. Then, taking logarithms  we obtain 
$$u_{t} = \tilde{X}_{t} -\frac{1}{2} \,  \langle \tilde{X} \rangle_{t} = -X_{t}+\frac{1}{2} \,  \langle X \rangle_{t}$$
and consequently %$\tilde{X}_{t} = -X_{t}+\langle X \rangle_{t}$  and 
$\langle u \rangle_{t}= \langle \tilde{X} \rangle_{t} =   \langle X \rangle_{t}$, which completes the proof. \\
\end{proof}

This  lemma has some important ramifications. First of all, from an application of Wald's identity it follows that for any stopping time $T$ with   $\Exp_{i}[\langle u \rangle_{\cT}]<\infty$ for $i=0,\infty$,
 \begin{equation} \label{klfap}
\frac{1}{2} \, \Exp_{\infty}[\langle u \rangle_{\cT}] = \Exp_{\infty} [ - u_{\cT}]= \Exp_{\infty} \Bigl[ \log \frac{d\Pro_{\infty}}{d\Pro_{0}}\Big|_{\cF_{\cT}} \Bigr]
\end{equation}
and that on the event $\{\cT > \tau\}$ 
\begin{align} \label{kldel}
\begin{split}
\frac{1}{2}  \,  \Exp_{\tau}[\langle u \rangle_{\cT}-\langle u \rangle_{\tau}  \, | \; {\mathcal{F}}_{\tau}] 
&= \Exp_{\tau} \Bigl[\bigl( u_{\cT}- u_{\tau} \bigr) \, \ind{\cT > \tau}  \, | \; {\mathcal{F}}_{\tau} \Bigr] \\
&= \Exp_{\tau} \Bigl[ \log \frac{d\Pro_{\tau}}{d\Pro_{\infty}}\Big|_{\cF_{\cT}}  \, \Big| \; {\mathcal{F}}_{\tau} \Bigr],
\end{split}
\end{align}
which highlights the connection of the performance measure $\cJ_{M}$  with the notion of Kullback--Leibler information. However, the most important consequence of the previous lemma is that it allows us to  obtain closed--form expressions for the performance characteristics of the CUSUM rule and establish its optimality,
following the methodology developed in \cite{moustito}. 

\begin{lemma} \label{L1}
Suppose that condition (\ref{full}) holds. For any stopping time $T$, $c>0$ and $\tau \in [0,\infty)$, 
\begin{equation} \label{fin}
\Pro_{\tau}(\tcus<\infty)=1, \quad \Pro_{\infty}(\tcus<\infty)=1,
\end{equation}
where $\tcus:=\cT \wedge \cus = \min(\cT,\cus)$. Moreover, on the event $\{\tcus > \tau\}$ we have
\begin{align} \label{perf}
\begin{split}
\Exp_{\tau}[\langle u \rangle_{\tcus} - \langle u \rangle_{\tau} |\cF_{\tau}] &= \Exp_{\tau}[g(y_{\tcus})-g(y_{\tau})| \cF_{\tau}] , \\
\Exp_{\infty}[\langle u \rangle_{\tcus} - \langle u \rangle_{\tau} | \cF_{\tau}] &= \Exp_{\infty}[h(y_{\tcus})-h(y_{\tau})| \cF_{\tau}],
\end{split}
\end{align}
where the functions $g$ and $h$ are defined  as follows:
\begin{equation} \label{gh}
g(x) :=e^{-x}+x-1, \quad  h(x):=e^{x}-x-1, \quad x \geq 0.
\end{equation}
\end{lemma}
% are the solutions to the following boundary value problems
%\begin{align} \label{gh}
%%\begin{split}
%&g''(x)+   g'(x)=-1, \quad x \in [0, c] ; \quad g(c)=0, \; g'(0)=0,\\
%& h''(x) -  h'(x)=-1, \quad x \in [0, c]; \quad h(c)=0, \; h'(0)=0.
%\end{split}
%\end{align}

\begin{proof}
Let us introduce the following notation:
$$\tcusn :=\tcus \wedge \inf\{t \geq 0: \langle u \rangle_{t} \geq n\}, \quad n \in \mathbb{N}.$$
From an application of It\^{o}'s rule we have:
\begin{align*} %\label{ito}
g(y_{\tcusn})-g(y_{\tau}) &= \int_{\tau}^{\tcusn} g'(y_{s}) dy_{s} + \frac{1}{2} \int_{\tau}^{\tcusn} g''(y_{s}) d\langle y \rangle_{s}.
\end{align*}
From Lemma \ref{L0} and the definition of the CUSUM statistic $\{y_{t}\}$ we have  
$$y_{t}= u_{t}-m_{t}= -X_{t}+ \frac{1}{2} \, \langle u \rangle_{t} -  m_{t}, \quad m_{t}:=\inf_{0 \leq s \leq t} u_{s},$$
and consequently $\langle y \rangle_{t}=  \langle u \rangle_{t}=  \langle X \rangle_{t}$, therefore we can write 
\begin{align*} % \label{ito2}
g(y_{\tcusn})-g(y_{\tau}) &= \int_{\tau}^{\tcusn} g'(y_{s}) \Bigl[-dX_{s}+ \frac{d\langle u \rangle_{s}}{2} - dm_{s}\Bigr]+ \frac{1}{2} \int_{\tau}^{\tcusn} g''(y_{s}) d\langle u \rangle_{s}. 
\end{align*}
Now, with a re-arrangment in its right-hand side and using the fact that the measure $dm_{s}$ is carried by the set $\{y_{s}=0\}$,  
the previous relationship takes the following form:
\begin{align*}
g(y_{\tcusn})-g(y_{\tau}) &= \frac{1}{2}  \int_{\tau}^{\tcusn} (g'+g'')(y_{s}) \,  d\langle u \rangle_{s} -\int_{\tau}^{\tcusn} g'(y_{s}) dX_{s} - \int_{\tau}^{\tcusn} g'(0) \,dm_{s}.
\end{align*}
From (\ref{gh}) it is clear that $g'(x)+g''(x)=1$ for every $x \geq 0$ and $g'(0)=0$, thus
\begin{align*}
\begin{split}
g(y_{\tcusn})-g(y_{\tau}) &= \frac{1}{2} \, \int_{\tau}^{\tcusn} d\langle u \rangle_{s} -\int_{\tau}^{\tcusn} g'(y_{s}) dX_{s}.
\end{split}
\end{align*}
Taking conditional expectation under $\Pro_{\tau}$ given $\cF_{\tau}$, on the event $\{\tcus > \tau\}$ we have:
\begin{align} \label{uuu}
\Exp_{\tau} \Bigl[g(y_{\tcusn})-g(y_{\tau})|\cF_{\tau} \Bigr] &= \Exp_{\tau} \Bigl[ \frac{1}{2} \, \int_{\tau}^{\tcusn} d\langle u \rangle_{s} \Big| \cF_{\tau} \Bigr] - \Exp_{\tau} \Bigl[\int_{\tau}^{\tcusn} g'(y_{s}) dX_{s} \Big| \cF_{\tau} \Bigr].
\end{align}
Since  $0 \leq y_{s} \leq c$ on  $\{ \tau  < s < \tcusn \leq \tcus \}$ and $g'$ is an increasing function on $[0,\infty)$, we have $g'(y_{s})\leq g'(c)$. Moreover, by the definition of the stopping time $\tcusn$ it is clear that $\langle X \rangle_{\tcusn}=\langle  u \rangle_{\tcusn} \leq n$, therefore   
$$\int_{\tau}^{\tcusn} (g'(y_{s}))^{2} \, d \langle  X \rangle_{s} \leq (g'(c))^{2} \, \langle  X \rangle_{\tcusn} \leq (g'(c))^{2} \, n< \infty.$$
Then, since $\Pro_{\tau}$ coincides with  $\Pro_{0}$ on $\cF_{\tau}$ and $X$ is local martingale (that starts from 0) under $\Pro_{0}$, 
the second term in the right hand side of (\ref{uuu}) vanishes and consequently we can write (on the event $\{\tcus > \tau\}$): 
\begin{align} \label{ooo}
\Exp_{\tau}[g(y_{\tcusn})-g(y_{\tau})|\cF_{\tau}]=  \Exp_{\tau} \Bigl[ \frac{1}{2} \, \int_{\tau}^{\tcusn} d\langle u \rangle_{s} \Big| \cF_{\tau} \Bigr].
\end{align}
Since $g$ is an increasing function on $[0,\infty)$ and $y_{\tcusn} \in [0,c]$, it is clear  that  the left-hand side of this equality is bounded by $g(c)-g(0)=g(c)$. Moreover, due to condition (\ref{full}), $\tcusn$ converges to $\tcus$ $\Pro_{0}$--a.s. as $n \rightarrow \infty$. Therefore, letting $n\rightarrow \infty$ in the right-hand side of (\ref{ooo}) and applying the Monotone Convergence Theorem, which we can do since $\langle  u \rangle_{t}$ is increasing, we obtain
\begin{align} \label{yyy}
g(c) &\geq \Exp_{\tau}\Bigl[  \frac{1}{2} \, \int_{\tau}^{\tcus} d\langle u \rangle_{s} \Big| \cF_{\tau}\Bigr] \geq \Exp_{\tau} \Bigl[ \frac{1}{2} \, \ind{\tcus =\infty}  \int_{\tau}^{\infty} d\langle u \rangle_{s} \Big|\cF_{\tau} \Bigr]. 
\end{align}
From condition (\ref{full}) and Lemma \ref{L0} it follows that $\Pro_{0} (\int_{\tau}^{\infty} d\langle u \rangle_{s}=\infty)=1$  and consequently $\Pro_{\tau}(\int_{\tau}^{\infty} d\langle u \rangle_{s}=\infty)=1$
since by definition $\Pro_{\tau}$ coincides with $\Pro_{0}$ on $\cFt$ for $t\geq \tau$. Then, from the law of iterated expectation we have  
$$\Exp_{\tau} \Bigl[\Pro_{\tau} \Bigl(\int_{\tau}^{\infty} d\langle u \rangle_{s}=\infty \Big|\cF_{\tau} \Bigr)\Bigr]=1,$$
which implies $\Pro_{\tau}(\int_{\tau}^{\infty} d\langle u \rangle_{s}=\infty |\cF_{\tau})=1$ and consequently $\Pro_{\tau}(\tcus =\infty)=0$, otherwise the right--hand side in (\ref{yyy}) becomes infinite, due to condition (\ref{full}), which leads to a contradiction. 

Now, going back to (\ref{ooo}) and letting $n \rightarrow \infty$ in both sides simultaneously, from an application of Bounded (Monotone) Convergence Theorem for the left (right)--hand side we obtain
%\begin{align} \Exp_{\tau}[g(y_{\tcus})-g(y_{\tau})|\cF_{\tau}]  &=  \Exp_{\tau} \Bigl[\int_{\tau}^{\tcus} d\langle u \rangle_{s} |\cF_{\tau} \Bigr] \end{align}
%since $\tcusn \rightarrow \tcus$ as $n \rightarrow \infty$, % Since $\langle u \rangle_{t}$  is an increasing function, it follows that $$(\langle u \rangle_{\tcus} - \langle u \rangle_{\tau})^{+}....$$
the first relationship in (\ref{gh}). The second relationship in (\ref{gh}) and the fact that  $\Pro_{\infty}(\cus =\infty)=0$ can be shown in a similar way, which completes the proof. \\
\end{proof}

When we set $\cT=\cus$ in the previous lemma, we obtain some very interesting information regarding the behavior of the CUSUM test. First of all, from (\ref{fin}) it follows that $\cus$ terminates almost surely for any change--point, $\tau$, as well as when the change never occurs.  Therefore, it is almost sure that the CUSUM test will raise a false alarm, and this explains why a constraint is imposed on the \textit{rate} of false alarms, not the probability that a false alarm will occur.  

Furthermore, from the first relationship in (\ref{perf}) it follows that  on the event $\{\cus >\tau\}$ we have
\begin{equation*} 
\frac{1}{2} \, \Exp_{\tau}[ \langle u \rangle_{\cus} - \langle u \rangle_{\tau} |\cF_{\tau}] = \Exp_{\tau}[g(y_{\cus})-g(y_{\tau})| \cF_{\tau}]= g(c)-g(y_{\tau}),
\end{equation*}
since $y_{\cus}=c$, due to the path--continuity of the CUSUM process, $\{y_{t}\}$. Therefore, since $y_{\tau}\in [0,c]$ on $\{\cus >\tau\}$ and $g$ is  increasing function on $[0,\infty)$, it follows that
\begin{align} \label{percus}
\begin{split}
\cJ_{M}[\cus] &= \sup_{ \tau \geq 0} \; \text{esssup}  \; \frac{1}{2} \, \Exp_{\tau} [\bigl(\langle u \rangle_{\cus}-\langle u \rangle_{\tau} \bigr) \, \ind{\cus >\tau} \, | \; {\mathcal{F}}_{\tau}] \\
              &= \sup_{ \tau \geq 0} \; \text{esssup}  \; \Exp_{\tau} [ (g(c)-g(y_{\tau})) \, \ind{\cus >\tau} \,| \; {\mathcal{F}}_{\tau}]=g(c).
\end{split}              
\end{align}
In other words, the worst-case scenario for the CUSUM test occurs when the change occurs at time $\tau=0$, i.e. $\cJ_{M}[T]= \frac{1}{2} \, \Exp_{0}[\langle u \rangle_{T}]$.

Finally, setting $\tau=0$ in the second relationship in (\ref{perf}) we obtain  
$$ \frac{1}{2} \, \Exp_{\infty}[\langle u \rangle_{\cus}]= \Exp_{\infty}[h(y_{\tcus})]=h(c).$$
Therefore, for the CUSUM test to satisfy the false alarm constraint with equality, its threshold $c$ should be chosen 
as the (unique) solution to the non--linear equation $h(c)=\gamma$. The following lemma shows that, without loss of generality, we can actually 
restrict ourselves to stopping times that satisfy the  false alarm constraint with equality.

\begin{lemma} \label{L2}
Suppose that condition (\ref{full}) holds. For any stopping time  $T$ such that  $\frac{1}{2} \,\Exp_{\infty}[\langle u \rangle_{\cT}]> \gamma$, there exists some $c>0$ so that  $  \frac{1}{2} \, \Exp_{\infty}[\langle u \rangle_{\tcus}]= \gamma$ and $\cJ_{M}[\tcus] \leq \cJ_{M}[T]$.
\end{lemma}

\begin{proof}
Consider an arbitrary stopping time with  $ \frac{1}{2} \, \Exp_{\infty}[\langle u \rangle_{\cT}]> \gamma$ and define the function $\psi(c):=  \frac{1}{2} \,\Exp_{\infty}[\langle u \rangle_{\tcus}]$. From the second relationship in \eqref{perf} it follows that  $\psi(c) = \Exp_{\infty}[h(y_{\tcus})]$, which implies that $\psi(c)$ is a continuous function. Then, since 
$\psi(0)=0$, $\psi(\infty)= \frac{1}{2} \, \Exp_{\infty}[\langle u \rangle_{\cT}]>\gamma$, there exists some $c>0$ so that $\psi(c)= \gamma$. By definition, $\tcus \leq \cT$, which implies $\cJ_{M}[\tcus] \leq \cJ_{M}[T]$ and completes the proof. \\
% thus it remains to show that $\cJ_{M}[\tcus] \geq \cJ'_{M}[\tcus]$.% since  by definition $\tcus \leq \cT$. % it is clear that  $\cJ_{M}[\tcus] \leq \cJ_{M}[T]$ which proves the claim.  (c) This can now be shown as in \cite{moustito}, Theorem 2.
\end{proof}

%Lemma \ref{L2} implies that we only need to consider stopping times that satisfy the false alarm constraint with equality.
%, that is $\Exp_{\infty}[\langle u \rangle_{T}]=\gamma$. Moreover, setting $\cT=\cus$ in (\ref{perf2}) it follows that the threshold of the CUSUM test should be chosen so that $\gamma=\Exp_{\infty}[\langle u \rangle_{\cus}]= h(y_{\cus})=h(c).$ 
%We can now provide a sketch for the proof of Theorem \ref{MAIN}.

%from (\ref{condinf}) we have $\Pro_{\infty}$-a.s. on $\{\cus >t\}$:
%\begin{equation} \label{perf}
%  =   \Exp_{\infty}[ A_{\cus} - A_{t} | \cFt] = g(-c)-g(-y_{t}).
%\end{equation} 
%The second part in (\ref{perfor}) now follows by setting $t=0$ in (\ref{perf}). 
%In order to prove (\ref{moustach}), it suffices to show that $\Pro_{\tau}$-a.s. on $\{\cus \geq \tau\}$ 
%\begin{equation*}
%\Exp_{\tau} [(u_{\cT}-u_{\tau}) | \; {\mathcal{F}}_{\tau} ]  = \Exp_{\tau} [A_{\cT}-A_{\tau} | \; {\mathcal{F}}_{\tau} ] 
%\end{equation*}for any fixed $\tau \in [0, \infty]$. This follows from an application of optional sampling theorem  and a localization argument, 
%as long as the right-hand side is finite, which is indeed the case if $\Exp_{0}[A_{\cT}]< \infty$ and  $\Exp_{\infty}[A_{\cT}]<\infty$. 
%We now prove Theorem \ref{MAIN}. 

\begin{proof}[Proof of Theorem \ref{MAIN}]
The proof is based on the fact that for any stopping time $T$ and any $c>0$ we have the following lower bound
$$\cJ_{M}[\cT] \geq  \frac{\Exp_{\infty}[e^{y_{\tcus}} g(y_{\tcus})]}{\Exp_{\infty}[e^{y_{\tcus}}]},$$
which can be shown in exactly the same way as in Theorem 2 of \cite{moustito}. From (\ref{percus}) and the  fact that  $y_{\cus}=c$ almost surely  
it follows that both sides of the inequality become equal to $g(c)$ when we set $\cT=\cus$. Due to this observation and Lemma \ref{L2},
it suffices to show that 
$$g(c) = \inf_{T} \frac{\Exp_{\infty}[e^{y_{\tcus}} g(y_{\tcus})]}{\Exp_{\infty}[e^{y_{\tcus}}]} \quad \text{when} \quad \frac{1}{2} \, \Exp_{\infty}[\langle u \rangle_{T}]=\gamma.$$
This can be done in exactly the same way as in Theorem 3 of \cite{moustito}.  
\end{proof}

\subsection{The case of diffusion-type processes} \label{dt}
Let us now illustrate Theorem \ref{MAIN} in the case that the observed process $\{\xi_{t}\}$ is a standard Brownian motion  before the change and adopts a (random, in general) drift after the change. More specifically, let $\Pro_{\infty}$ be the Wiener measure and suppose that the  post-change measure $\Pro_{0}$ is induced by the following dynamics 
\begin{equation*} %\label{sdeY000}
\xi_{t}= W_{t}+ \int_{0}^{t} \mu_{s} \, ds, \quad t \geq 0
\end{equation*}
where $\{W_{t}\}$ is a standard Brownian motion under $\Pro_{0}$ and $\{\mu_t\}$ an $\{\cFt\}$--adapted process
% which justifies the term \textit{diffusion--type}, 
that satisfies 
\begin{align}  
& \Pro_{\infty} \Bigl( \int_{0}^{t}  \mu_{s}^{2} \, ds < \infty  \Bigr)=1,  \quad t \geq 0 , \label{con1}\\
&  \Exp_{\infty} \Bigl[ \exp \Bigl\{ \int_{0}^{t}  \mu_{s} \,  d\xi_{s} - \frac{1}{2} \, \int_{0}^{t} \mu_{s}^{2} \, ds \Bigr\} \Bigr] =1, \quad t \geq 0. \label{con2}
\end{align}
These two conditions  guarantee that $\Pro_{0}$ is indeed well--defined and locally equivalent with $\Pro_{\infty}$. Moreover, from Girsanov's theorem we obtain the following explicit representation for their log-likelihood ratio, 
\begin{equation*} %\label{uito}
u_{t}= \int_{0}^{t}  \mu_{s} \,  d\xi_{s} - \frac{1}{2} \, \int_{0}^{t} \mu_{s}^{2} \, ds , \quad t \geq 0.
\end{equation*}
It is then clear that $\langle u \rangle_{t} =  \int_{0}^{t} \mu_{s}^{2} \, ds$ and Theorem \ref{MAIN} implies that 
the CUSUM test is $\cJ_{M}$--optimal if the following condition is also satisfied
\begin{equation}
\int_{0}^{\infty} \mu_{s}^{2} \, ds  = \infty, \quad \Pro_{0}, \Pro_{\infty}-\text{a.s.}  \label{fullito}
\end{equation}
This is exactly the result obtained in \cite{moustito}. When the process $\{\mu_{t}\}$ reduces to a constant, $\langle u \rangle_{t}$ is proportional to $t$
and Theorem \ref{MAIN} implies that the CUSUM test is $\cJ_L$--optimal, which is also a well--known fact. 
 
%However, there are many interesting continuous-path stochastic processes that are not of diffusion-type, thus

However,  Theorem \ref{MAIN} is a non--trivial generalization of existing optimality results, as there are many continuous--path processes that are not of diffusion--type, not even semimartingales. We  illustrate this point in the next section.

% we consider the case of \textit{fractional} diffusion-type processes.  
%This is exactly the problem considered in Moustakides in \cite{moustito}, where it was also shown that \eqref{fullbm} 
%is satisfied in the Ornstein-Uhlenbeck case, that is when $\mu_{t}= \theta  \xi_{t}$, where $\theta$ is some non-zero constant,
%and it is straightforward to see that \eqref{fullbm} is satisfied when $\mu_{t}= t^{\alpha}$ with $\alpha >-1/2$.  
%Furthermore, we will show that the CUSUM test optimizes Lorden's original criterion for a class of processes that includes BM.
% was shown in  showed that for a class of It\^{o} processes that satisfy a ``full--energy'' condition, is $\cJ_{M}$--o ptimal, i.e. for every $\gamma>0$ there exists a threshold $c$ (that depends on $\gamma$) so that \begin{equation}   \label{cusopt} \cJ_{M}[\cus] = \inf_{\cT \in \cC_{M}(\gamma)}  \cJ_{M}[\cT].
%\end{equation} A direct corollary of this result is that the CUSUM test $\cus$ is $\cJ_{L}$--optimal for detecting a \textit{linear} drift in a \textit{standard} Brownian motion, a result that had been shown earlier by Shiryaev and Beibel. 

\section{The case of fractional diffusion-type processes}
Before we consider the application of Theorem \ref{MAIN} in the context of fractional diffusion--type processes, 
let us first define \textit{fractional Brownian motion}, the basic building block of these processes,  and present its main properties. 

\subsection{Fractional Brownian Motion: A Quick Review} 
Let $\{B_{t}^{H}\}$ be a stochastic process defined on the canonical space $(\Omega, \cF)$ and  let $\Pro$ be a probability measure on this space. 
We will say that  $\{B_{t}^{H}\}$ is a \textit{fractional Brownian motion} (fBm) under $\Pro$, or simply a $\Pro$-fBm, if it is a centered, continuous, Gaussian process  with covariance structure 
\begin{equation*} %\label{fBmDef}
\Exp [B^{H}_{t} \, B_{s}^{H}] = \frac{1}{2} \left( t^{2H} + s^{2H} -|t-s|^{2H} \right), \quad t,s \geq 0,
\end{equation*} 
where $\Exp$ refers to expectation with respect to $\Pro$. As a consequence of its definition, the process $\{B_{t}^{H}\}$ has a number of interesting properties. Thus,

\begin{enumerate}
\item[(i)]  it is $H$-self-similar,  in the sense that $\{B_{t}^{H}\}_{t \geq 0}$ has the same finite-dimensional distributions as  $\left\{c^{-H}\; B^{H}_{ct} \right\}_{t \geq 0}$ for every $c>0$,
\item[(ii)] it has  stationary  increments, which however are independent only when $H=1/2$. When $H<(>)1/2$, they are  negatively (positively) correlated and the process exhibits \textit{short(long)--range dependence}, in the sense that 
$$\sum_{n=1}^{\infty} \Exp[(B^{H}_{n} - B^{H}_{n-1}) B^{H}_{1}]< (=) \infty,$$
\item[(iii)] it has H\"older continuous paths of order $H-\epsilon$ for every $0 < \epsilon <H$, 
% in the sense that for all $T>0$, there exists a nonnegative random variable $C_{\epsilon, T}$, which is  $L^{p}$-integrable for all $p\geq 1$, so that: $$|B^{H}_{t} - B^{H}_{s}| \leq C_{\epsilon, T} \, |t-s|^{H-\epsilon}, \quad \forall s,\;t\in[0,T] \quad  \Pro-a.s.$$ 
\item[(iv)] it has finite $1/H$-variation (in an $\text{L}^{1}$ sense), equal to $c'_{H} t$ on any finite interval $[0,t]$, where 
$$c'_{H}:= \int_{\mathbb{R}} |x|^{1/H} \phi(x) dx, \quad \phi(x):= \frac{1}{\sqrt{2 \pi}} \, e^{-x^{2}/2}.$$
%with $\phi(x)$ representing the probability density function of a standard normal distribution. 
\item[(v)] is \textit{not} a semimartingale, i.e. it does not admit a decomposition as the sum of a local martingale and a term of finite variation  (see, for example, \cite{rogers}). However, the transformed process  
\begin{align}  \label{M}
 \begin{split}
M_{t}^{H}  &:= \int_{0}^{t} k_{H}(t,s) \, dB^{H}_{s} , \quad t \geq 0 
\end{split}
\end{align}
is a square--integrable martingale with  quadratic variation  $\langle M^{H} \rangle_{t}= \lambda^{-1}_{H} \,  t^{2-2H}$, where 
$$ k_{H}(t,s) := c_{H}^{-1} \;s^{\frac{1}{2}-H}\;(t-s)^{\frac{1}{2} - H}, \quad 0 \leq s \leq t, $$
$c_{H}$ and $\lambda_{H}$ are positive constants defined as follows
\begin{align*} 
c_{H} &:= 2H\;\Gamma\left( 3/2- H \right)\; \Gamma\left(H + 1/2 \right) ,   \\
\lambda_{H} &:= \frac{2H\;\Gamma(3-2H)\;\Gamma\left(H + 1/2 \right)}{\Gamma\left( 3/2 - H \right)},
\end{align*} 
and $\Gamma(x):= \int_{0}^{\infty} s^{x-1} e^{-s} \, ds$ is the  Gamma function. 
This result  was shown by Molchan \cite{molchan} and more recently by Norros et. al \cite{NVV}, 
where the process $M^{H}$ was called the \textit{fundamental martingale associated with fBm}. 
\end{enumerate}

For an exhaustive  treatment on the properties of fBm we refer to Chapter 5 of Nualart \cite{N}. Here, we will only add an extension of L\'evy's classical characterization theorem to fBm that was recently established by Hu et al. \cite{hu} (see also Mishura and Valkeila \cite{mis}). This result essentially says that properties (iii), (iv) and (v) characterize fBm. More specifically, if $Y$ is a continuous, centered, square-integrable stochastic process with 
\begin{itemize}
\item[(a)] H\"older continuous paths of order $H-\epsilon$ for any $\epsilon >0$,
\item[(b)] finite $1/H$--quadratic variation (in an $L^{1}$--sense) that is equal to $c'_{H} t$ on any interval $[0,t]$,
\item[(c)] and the process $M$,  defined as $M^{H}$ in (\ref{M}) with $B^{H}$ replaced by $Y$,  is a martingale, 
whose quadratic variation is absolutely continuous almost surely with respect to the Lebesgue measure when $H>1/2$,
\end{itemize}
then $Y$ is a fBm with Hurst index $H$.

%Based on this characterization and the classical Girsanov theorem, we can obtain sufficient conditions on the drift $\{\mu_{t}\}$ so that the post-change measure $\Pro_{0}$ is well--defined and the CUSUM test is optimal in this context. 

\subsection{The case of a fBm adopting a random drift}
We now return to the setup of our change detection problem and we focus on the special case that the observed process $\{\xi_{t}\}$ is a fBm that adopts a random drift after the change. More specifically, in what follows we assume that $\{\xi_{t}\}$ is a fBm with Hurst index H under $\Pro_{\infty}$, whereas 
$\Pro_{0}$ is induced by the following dynamics
\begin{align} \label{sdeY0}
\xi_{t} &= B_{t}^{H}+\int_{0}^{t} \mu_{s} \, ds, \quad t \geq 0,
\end{align}   
where $\{\mu_{t}\}$ is an $\{\cFt\}$--adapted process and $\{B_{t}^{H}\}$ a fBm  under $\Pro_{0}$ with the same Hurst index, $H$.

 When $H=1/2$, we recover the context of Subsection \ref{dt}, where we saw that  $\Pro_{0}$ is indeed  well--defined and locally equivalent to $\Pro_{\infty}$ when conditions (\ref{con1}) and (\ref{con2}) are satisfied and that the CUSUM test is $\cJ_{M}$--optimal if additionally condition (\ref{fullito}) holds. Our goal in this section is to obtain analogous conditions when $H \neq 1/2$, in which case $\{\xi_{t}\}$ is \textit{not} a semimartingale and the classical Girsanov theorem does not apply, thus the previous conditions are no longer appropriate. 
 
In order to do so, we will work with the transformed process
\begin{align}  \label{ZZZ}
\zeta_{t} &:= \int_{0}^{t} k_{H}(t,s) \, dB^{H}_{s} , \quad t \geq 0.
\end{align}
Since $\{\xi_{t}\}$ is a fBM under $\Pro_{\infty}$, from property (v) it follows that $\{\zeta_{t}\}$ is a square-integrable martingale with quadratic variation  $\langle \zeta \rangle_{t}=  \lambda^{-1}_{H} \,  t^{2-2H}$ under $\Pro_{\infty}$. This property allows us to obtain an explicit representation for the log-likelihood ratio process $\{u_{t}\}$ using a version of Girsanov's theorem, although the observed process $\{\xi_{t}\}$ itself is not a semimartingale. 
This is a well--known result in the Stochastic Calculus with respect to fBm (see \cite{DU}, \cite{KLB3}). Here,  we will provide a  proof that relies on the 
characterization of fBm that we previously discussed. 

In what follows,  we assume that the paths of $\{\mu_{t}\}$ are sufficiently smooth so that the process
\begin{equation} \label{Q}
Q_{t}:= \frac{d}{d\langle \zeta\rangle_{t}}\int_{0}^{t}k_{H}(t,s) \, \mu_{s} \, ds, \quad t \geq 0
\end{equation}
is well--defined, where the derivative is understood in the sense of absolute continuity.

\begin{theorem} \label{main}
If the following conditions hold, 
\begin{align*}
& \Pro_{\infty} \Bigl( \int_{0}^{t} Q^{2}_{s} \, d\langle \zeta \rangle_{s} < \infty \Bigr)=1, \quad t \geq 0,   \\
& \Exp_{\infty} \Bigl[ \exp \Bigl\{  \int_{0}^{t} Q_{s} \, d\zeta_{s} - \frac{1}{2}  \int_{0}^{t} Q^{2}_{s} \, d\langle \zeta \rangle_{s} \Bigr\} \Bigr] =1, \quad t \geq 0,
\end{align*}
then $\Pro_0$ is locally equivalent to $\Pro_{\infty}$ and their log-likelihood ratio process admits the following representation
\begin{equation} \label{llrfbm2}  %\frac{\text{d} \Pro_{\tau}}{\text{d}\Pro_{\infty}} \Big|_{\cFt} 
u_{t}=  \int_{0}^{t} Q_{s} \, d\zeta_{s} - \frac{1}{2}  \int_{0}^{t} Q^{2}_{s} \, d\langle \zeta \rangle_{s} , \quad  t \geq 0.
\end{equation}
Then, the CUSUM test is $\cJ_{M}$--optimal, if additionally the following condition is satisfied
\begin{align} \label{fullQ}
\int_{0}^{\infty} Q_{s}^{2} \, d\langle \zeta \rangle_{s} = \infty \quad  \Pro_{0}, \Pro_{\infty}-\text{a.s.} 
\end{align}
\end{theorem}

\vspace{0.2cm}

\begin{proof}
Due to Theorem \ref{MAIN}, we only need to show that the log--likelihood ratio $\{u_{t}\}$ admits representation (\ref{llrfbm2}). In order to do so, 
it suffices to show that if the post-change  measure $\Pro_{0}$ is defined by 
\begin{equation} \label{des}
\frac{\text{d} \Pro_{0}}{\text{d}\Pro_{\infty}} \Big|_{\cFt}= \exp \Bigl\{ \int_{0}^{t} Q_{s} \, d\zeta_{s} - \frac{1}{2}  \int_{0}^{t} Q^{2}_{s} \, d\langle \zeta \rangle_{s} \Bigr\}, \quad  t \geq 0,
\end{equation}
then the process
\begin{equation*} % \label{sdeY10}
B^{H}_{t}:= \xi_{t}-\int_{0}^{t} \mu_{s} \, ds, \quad  t \geq 0,
\end{equation*}
is a fBm under $\Pro_{0}$. Moreover, due to the characterization theorem that we discussed in the end of the previous subsection, it suffices to show that $\{B^{H}_{t}\}$ satisfies properties (a), (b) and (c). Since  $\{B^{H}_{t}\}$ is a ``shifted'' version of $\{\xi_{t}\}$, which is a fBm under $\Pro_{\infty}$, 
it clearly satisfies (a) and (b). It  remains to show that the process $\{\int_{0}^{t} k_{H}(t,s) \, dB_{s}^{H}\}$ is a $\Pro_{0}$--martingale whose quadratic variation is absolutely continuous with respect to the Lebesgue measure when $H>1/2$. Indeed, 
\begin{align} \label{sdeY2222}
\begin{split}
 \int_{0}^{t} k_{H}(t,s) \; d B_{s}^{H} &=  \int_{0}^{t} k_{H}(t,s) \, (d\xi_{s}- \mu_{s} \, ds) \\
 &=\zeta_{t}-  \int_{0}^{t} Q_{s} \, d\langle \zeta\rangle_{s}                                                             
\end{split}          
\end{align} 
where the second equality follows from the definitions of $\zeta$ and $Q$ in (\ref{ZZZ}) and (\ref{Q}), respectively. But from (\ref{des}) and Girsanov's theorem we know that the process in the right--hand side of (\ref{sdeY2222}) is a $\Pro_{0}$--martingale with quadratic variation $\langle \zeta \rangle_{t}=  \lambda^{-1}_{H} \,  t^{2-2H}$, 
which completes the proof. 
\end{proof}

In what follows, we apply this result in some interesting special cases.

\subsubsection{Fractional diffusions}
The conditions of Theorem \ref{main} are satisfied for a large class of fractional diffusions and any Hurst index $H$. This is the content of the following corollary, which is based on the results of Tudor and Viens \cite{TV}, Lemma 3. 

\begin{corollary}
The CUSUM is $\cJ_{M}$--optimal when $\mu_{t}= b(\xi_{t})$ and $b$ is a real function so that
\begin{enumerate}
\item[(i)] $x b(x)$ has a constant sign for all $x \geq 0$ and a constant sign for all $x \leq 0$,
\item[(ii)] $|b(x)/x|= c+ r(x)$ for all $x$, where $r(x) \rightarrow \infty$ as $x\rightarrow  \infty$. 
%  there exists a function $h(x)$ with $\lim_{x \rightarrow \infty} h(x)=0$ so that $\left| \frac{b(x)-c_{0}}{x}\right| = c_{1} + h(x)$ as $x\rightarrow \pm \infty$, where $c_{0}$, $c_{1}$ are arbitrary ? constants as in Proposition \ref{viensDrift} ???. 
\end{enumerate}
\end{corollary}
These two conditions are satisfied when for example  $b(x)= c_{0} + c_{1}x + (|x| \wedge 1)^{\alpha}$ with  $c_{0}, c_{1} \in \mathbb{R}$ and $\alpha \in [0,1)$.
When in particular $\alpha=0$, $b(x)$ is an affine function and, under the post--change measure $\Pro_{0}$,  $\{\xi_{t}\}$ is a \textit{fractional} Ornstein--Uhlenbeck  process. This is the fractional analogue of the classical Ornstein--Uhlenbeck process and has been used in financial modeling (see for example \cite{CV}, \cite{CR}). For a study of its main properties we refer to  Cheridito et al. \cite{che}.  
%In this case, it can be shown using the results of  Kleptsyna and Le-Breton \cite{KLB1,KLB2} that the conditions of Theorem \ref{main} are always satisfied, which  implies that the CUSUM test is optimal for detecting a change from a fBm to fOU for \textit{any} Hurst index $H$.
%(\ref{sdeY10}) reduces to 
%\begin{equation} \label{fdiff} \xi_{t} = B^{H}_{t} +  \;  \theta  \, \int_{0}^{t} \xi_{s}  \, ds,
%\quad t \geq 0 ,
%\end{equation} 
%and condition (\ref{fullQ}) takes the following form:
%\begin{equation}
%\int_{0}^{\infty} \frac{d}{dm_{t}} \, \int_{0}^{t}  k_{H}(t,s) \, \xi_{s} \,  ds = \infty, \quad \Pro_{0}, \Pro_{\infty}-\text{a.s.}
%\end{equation}
%\begin{align*}  Q_{t} &= \frac{d}{dm_{t}} \, \int_{0}^{t}  k_{H}(t,s) \, \xi_{s} \,  ds.   \end{align*} More generally, t can be shown that the conditions of Theorem \ref{main} are satisfied when  which implies that the optimality of the CUSUM test is not limited to linear, Gaussian fractional diffusions. 

%-----------------------------------------------------------
%-----------------------------------------------------------
%-----------------------------------------------------------

\subsubsection{Polynomial drift}
When $\mu_{t}=t^{\alpha}$, where $\alpha$ is a real constant,  the post-change dynamics (\ref{sdeY0}) reduce to 
\begin{align*} %\label{fbmdrift}
\xi_{t} &=  B^{H}_{t} + \int_{0}^{t} s^{\alpha} ds = B^{H}_{t} + \frac{t^{\alpha+1}}{\alpha+1}
\end{align*}
and the process $Q_{t}$ takes the following form
\begin{align*} % \label{QHa}
Q_{t} %&= \frac{d}{d\langle \zeta \rangle_{t}} \, \int_{0}^{t} s^{\alpha}  \, k_{H}(t,s) \, ds  \nonumber \\
& = \frac{1}{c_{H}} \; \frac{\text{d}}{\text{d} \langle \zeta \rangle_{t}} \, \int_{0}^{t} \, s^{\frac{1}{2}-H+\alpha} \,  (t-s)^{\frac{1}{2}-H} \, ds. %d_{H,\alpha} \, t^{\alpha},
\end{align*}
After some algebraic manipulations %that include integrals of the Beta function,
%\begin{equation} \label{beta}B(x,y):= \int_{0}^{1} t^{x-1} \, (1-t)^{y-1} \, dt = \frac{\Gamma(x) \,  \Gamma(y)}{\Gamma(x+y)}, \quad x,y >0,\end{equation}
we  obtain  $Q_{t}= d_{H,\alpha} \, t^{\alpha}$, where
\begin{align*}  %\label{dHa} 
d_{H,\alpha} &:= \frac{\Gamma(3-2H) \; \Gamma(3/2-H+ \alpha)}{\Gamma(3-2H+\alpha) \; \Gamma(3/2-H)} \, \frac{2-2 H+\alpha}{2-2H},
\end{align*}
%where   $d_{H,\alpha}$ is a positive constant that depends both on H and $\alpha$,
%\begin{align}  d_{H,\alpha} &:= \frac{B(3/2-H + \alpha , 3/2-H)  \,\Gamma(3-2H)}{c_{H}^{2}} \, \frac{2-2 H+\alpha}{2-2H},   \label{dHa}  \end{align}
%where  $\alpha$  is some real constant. Since the kernel of the integral that of the Beta function, we obtain $Q_{t}= d_{H,\alpha} \, t^{\alpha}$ 
and consequently
\begin{equation} \label{qfbm}
\int_{0}^{t} Q_{s}^{2} \, d\langle \zeta \rangle_{s} =  v_{H,\alpha} \; t^{2-2H+2\alpha}, \quad v_{H,\alpha}:=  \frac{(d_{H,\alpha})^{2}}{\lambda_{H}}  \, \frac{1-H}{1-H+\alpha}.
\end{equation}

We can now state two interesting corollaries. 

\begin{corollary}   \label{coro1}
If $\mu_{t}=t^{\alpha}$ with $\alpha+1>H$, then the CUSUM test  is $\cJ_{M}$--optimal. 
\end{corollary}

\begin{proof}
%where  $v_{H,\alpha}$ is a positive constant that depends both on H and $\alpha$,
%\begin{align} 
%v_{H,\alpha} &:=  \frac{(d_{H,\alpha})^{2}}{\lambda_{H}}  \, \frac{1-H}{1-H+\alpha}. \label{vHa}
%\end{align} ^This discussion leads to some very interesting conclusions. First of all, it implies that the 
%CUSUM test solves Problem \ref{moust_crit} as long as $\alpha>H-1$, since it is clear 
From (\ref{qfbm}) it is clear that condition (\ref{fullQ}) is satisfied when  $\alpha>H-1$, thus the corollary follows directly from Theorem \ref{main}.
\end{proof}

\begin{corollary}  \label{coro2}
If $\mu_{t}=t^{\alpha}$ with $\alpha+1=H+1/2$, then  the CUSUM test is $\cJ_{L}$--optimal. 
\end{corollary}

\begin{proof}
From (\ref{qfbm}) it is clear that when $\alpha=H-1/2$,  $\langle u \rangle_{t}$ is proportional to $t$, thus the result follows from the previous corollary and Theorem \ref{MAIN}.
\end{proof}

One of the implications of Corollary \ref{coro1} is that the CUSUM test is always $\cJ_{M}$--optimal when a  \textit{linear} drift ($\alpha=0$) emerges in a fractional Brownian motion, no matter what the value of the Hurst index, $H$, is. 

Corollary \ref{coro2}  provides a class of processes for which the CUSUM test optimizes  Lorden's \textit{original} criterion. To our knowledge, such an optimality property has been established only in the case of \textit{diffusion--type} process for which the quadratic variation $\langle u \rangle _{t}$ is proportional to $t$ (see the discussion in pg. 313 of \cite{moustito}). Of course, the case of a linear drift emerging in a standard Brownian motion corresponds to the special case $H=1/2$ of Corollary \ref{coro2}. 

\subsection{Extensions}
It is possible to generalize Theorem \ref{main} to the case that the  pre--change measure $\Pro_{\infty}$ is induced by the following dynamics 
\begin{equation} \label{exte}
\xi_{t}= \int_{0}^{t} \sigma(s) \, d\tilde{B}^{H}_{s}, \quad t \geq 0,
\end{equation}
where $\tilde{B}^{H}$ is a $\Pro_{\infty}$--fBm and $\sigma: \mathbb{R}_{+} \rightarrow \mathbb{R}_{+}$ is a deterministic, non-vanishing real function  with  $\delta$--H\"older continuous paths for some  $\delta >1-H$. This smoothness condition guarantees that the integral in (\ref{exte}) can be defined in a
Riemann--Stieljes (or Young) sense (see \cite{young}). Then, the proof of Theorem \ref{main} goes through as long as we
modify the definitions of $\zeta$ and $Q$ in (\ref{ZZZ}) and (\ref{Q}) respectively as follows %%%%%%In this case, the transformed process $\zeta$ is defined as 
\begin{equation*} 
\zeta_{t}= \int_{0}^{t} k_{H}(t,s) \frac{1}{\sigma(s)} d\xi_{s}, \quad Q_{t}=\int_{0}^{t} k_{H}(t,s) \frac{\mu_{s}}{\sigma(s)} ds, \quad t \geq 0.
\end{equation*}
Note that when $\sigma$ is a stochastic process, the integral in (\ref{exte}) cannot be defined as a stochastic integral in an It\^{o} sense, 
since fBm is \textit{not} a semimartingale. %%%%%%For a definition that relies on the notion of Skorokhod integral we refer to \cite{DU}.

\section{Conclusions}
In this work, we extended the optimality properties of the  CUSUM test in two directions. 

First, with respect to a modified version of Lorden's criterion, where detection delay is measured in terms of the quadratic variation (that is accumulated between the time of the change and the time of stopping) of the log-likelihood ratio between the post and pre--change distribution. With respect to this criterion, the optimality of CUSUM was established for
arbitrary continuous-path processes, generalizing existing optimality results that refer to processes of diffusion--type. As an application, we
obtained sufficient conditions for the optimality of the CUSUM procedure when the observed process is a fractional Brownian motion 
that adopts a random drift. We saw that these conditions are satisfied in some interesting special cases, such as when a fractional Brownian motion 
turns into a fractional Ornstein--Uhlenbeck process or it adopts a linear drift, for any value of the Hurst  index.

Second, with respect to Lorden's original criterion, we proved that the CUSUM test is optimal when a fractional Brownian motion with Hurst index $H$ adopts a deterministic, polynomial drift term with exponent $H+1/2$. In this way, we generalized the well--known optimality of CUSUM in the case that a linear drift emerges in a standard ($H=1/2$) Brownian motion.

%Overall, this work provides a general framework  in which the optimal detection rule admits a simple structure and 
%the optimal performance a closed--form expression. The latter can serve as the ultimate benchmark in the case that 
%when one or more assumptions of the present framework are violated, such as when 
%the observed process is only sampled  at discrete times. % or when the pre and post--change distributions are only partially specified. 

%-----------------------------------------------------------
%-----------------------------------------------------------
%-----------------------------------------------------------

\end{document}